\documentclass[reqno]{amsart}

\usepackage{color}
\usepackage[dvipsnames]{xcolor}

\usepackage{ifpdf}
\ifpdf 
\usepackage[pdftex]{graphicx}   % to include graphics
\usepackage[pdftex,     % sets up hyperref to use pdftex driver
plainpages=false,   % allows page i and 1 to exist in the same document
breaklinks=true,    % link texts can be broken at the end of line
colorlinks=true,
linkcolor=red,
citecolor=green,
pdftitle=My Document
pdfauthor=My Good Self
]{hyperref} 
\else 
\usepackage{graphicx}       % to include graphics
\usepackage{hyperref}       % to simplify the use of \href
\fi 

\usepackage{subfig}
\usepackage{glossaries}
\usepackage{glossary-mcols}

\usepackage{aurical}
\usepackage{amsfonts,amsmath}	
\usepackage{amssymb}
\usepackage{verbatim}
\usepackage{amsopn}
\usepackage[english]{babel}
\usepackage{amsthm}
\usepackage{enumerate}
\usepackage{mathrsfs}	
\usepackage{mathtools}
\usepackage{esint}
\usepackage{todonotes}

\usepackage{bm}
\usepackage{bbm}
\usepackage{caption}

\theoremstyle{definition} \newtheorem{definition}{Definition}[section]
\theoremstyle{definition} \newtheorem{remark}[definition]{Remark}
\theoremstyle{plain} \newtheorem{lemma}[definition]{Lemma}
\theoremstyle{plain} \newtheorem{proposition}[definition]{Proposition}
\theoremstyle{plain} \newtheorem{theorem}[definition]{Theorem}
\theoremstyle{plain} \newtheorem{corollary}[definition]{Corollary}
\theoremstyle{definition} 
\theoremstyle{plain} 
\theoremstyle{definition}

\DeclareMathOperator{\BV}{BV}

\DeclareMathOperator{\dist}{dist}
\DeclareMathOperator{\supp}{supp}

\newcommand{\R}{\mathbb{R}}

\newcommand{\N}{\mathbb{N}}

\newcommand{\TV}{\text{\rm Tot.Var.}}

\newcommand{\e}{\varepsilon}

\newcommand{\loc}{\text{\rm loc}}
\newcommand{\D}{\mathcal{D}}

\newcommand{\M}{\mathcal{M}}

\renewcommand{\L}{\mathscr L}

\renewcommand{\S}{\mathbb S}
\newcommand{\Osc}{\mathrm Osc}
\newcommand{\h}{\mathrm height}

\renewcommand{\div}{\mathrm{div}}
\renewcommand{\L}{\mathscr L}
\renewcommand{\H}{\mathscr H}

\numberwithin{equation}{section} %numera le formule in accordo con le sezioni
\theoremstyle{plain} \newtheorem*{theorem*}{Theorem}
\theoremstyle{plain} 
\theoremstyle{plain} \newtheorem*{mthm*}{Main Theorem}
\theoremstyle{plain} \newtheorem*{conjecture*}{Conjecture}
\theoremstyle{plain} 
\theoremstyle{plain} \newtheorem*{problem*}{Problem}

\title{Differentiability properties of the flow of 2d autonomous vector fields}

\thanks{The author has been supported by the SNF Grant 182565.}

\author[E.~Marconi]{Elio Marconi}
\address{Elio Marconi, EPFL B, Station 8, CH-1015 Lausanne, CH.}
\email{elio.marconi@epfl.ch}

\begin{document}
	\maketitle
	
\begin{abstract}
We investigate under which assumptions the flow associated to autonomous planar vector fields inherits the Sobolev or BV 
regularity of the vector field. We consider nearly incompressible and divergence-free vector fields, taking advantage in both cases of the underlying Hamiltonian structure. Finally we provide an example of an autonomous planar Sobolev divergence-free vector field, such that the 
corresponding regular Lagrangian flow has no bounded variation.
\end{abstract}

\section{Introduction}
We consider bounded vector fields $b \in L^\infty \big((0,T)\times \R^d; \R^d\big)$.
Although the analysis of this paper is limited to the case of autonomous vector fields with $d=2$, we introduce the relevant notions and the related results in the general setting. The following notion of \emph{regular Lagrangian flow} is an appropriate extension for merely locally integrable vector fields of the classical flow associated to Lipschitz vector fields.

\begin{definition}
Given $b \in L^1_\loc((0,T)\times \R^d;\R^d)$, we say that $X:[0,T)\times \R^d\to \R^d$ is a \emph{regular Lagrangian flow} of the vector field $b$ if
	\begin{enumerate}
	\item for $\mathscr L^d$- a.e. $x\in \R^d$ the map $t\mapsto X(t,x)$ is absolutely continuous, $X(0,x)=x$ and for $\mathscr L^1$-a.e. $t\in (0,T)$ it holds $\partial_tX(t,x)=b(t,X(t,x))$;
	\item for every $t \in [0,T)$ it holds
	\begin{equation*}
	X(t,\cdot)_\sharp \mathscr L^d \le L\mathscr L^d,
	\end{equation*}
	for some $L>0$.
	\end{enumerate}
\end{definition} 

Regular Lagrangian flows have been introduced in a different form in \cite{DL_transport}, where the authors proved their existence and uniqueness for vector fields
$b \in L^1_tW^{1,p}_x$ with $p\ge 1$ and bounded divergence. 
The theory has been extended to vector fields $b \in  L^1_t \BV_x$ with bounded divergence in \cite{Ambrosio_transport}. 
Uniqueness of regular Lagrangian flows was finally achieved in the more general class of nearly incompressible vector fields with bounded variation in \cite{BB_Bressan}, introduced in the study of the hyperbolic system of conservation laws named after Keyfitz and Kranzer (see \cite{DL_notes}).
\begin{definition}\label{D_NI}
A vector field $b \in L^1_\loc ((0,T)\times \R^d;\R^d)$ is called \emph{nearly incompressible} if there exist $C>0$ and $\rho \in C^0([0,T); L^\infty_w(\R^d))$ solving the continuity equation
\begin{equation}\label{E_CE}
\partial_t \rho + \div_x(\rho b)=0
\end{equation}
with $\rho(t,x) \in [C^{-1},C]$ for $\L^{d+1}$-a.e. $(t,x)\in (0,T)\times \R^d$.
\end{definition}

Several results about the differentiability properties of regular Lagrangian flows are available now.
By the contributions in \cite{LL_differentiability, AM_differentiability}, it follows that regular Lagrangian flows associated to vector fields $b \in L^1_t W^{1,1}_x$ are \emph{differentiable in measure}
(see \cite{AM_differentiability} for the definition of this notion). The same regularity property has been obtained recently in \cite{BD_differentiability} for nearly incompressible vector fields with bounded variation.
The stronger property of \emph{approximate differentiability} was obtained in \cite{ALM_differentiability} for regular Lagrangian flows associated to vector fields $b \in L^1_t W^{1,p}_x$ with $p>1$. A quantitative version of the same regularity property was provided in \cite{CDL_DiPerna-Lions}, where the authors proved a quantitative Lusin-Lipschitz regularity of the flow.

The optimality of the regularity estimates obtained in \cite{CDL_DiPerna-Lions} is discussed in \cite{Jabin_example}.
In particular the author provided through a random construction an example of time dependent divergence-free Sobolev vector field in $\R^2$ such that the regular Lagrangian flow has not bounded variation.

\subsection{2d autonomous vector fields}
The analysis in the setting of 2d autonomous vector fields is facilitated by the following Hamiltonian structure: if $b \in L^\infty(\R^2;\R^2)$ with $\div \,b=0$, then there exists a Lipschitz Hamiltonian $H:\R^2 \to \R$ such that 
\begin{equation}\label{E_Hamiltonian}
b = \nabla^\perp H = (-\partial_2 H, \partial_1 H).
\end{equation}
At least formally the Hamiltonian is preserved by the flow, so that the trajectories of the flow are contained in the level sets of $H$.
In the series of papers \cite{ABC1,ABC2,ABC3}, the authors reduced the uniqueness problem for the continuity equation to a family of one-dimensional problems on the level sets of $H$. 
With this approach they were able to characterize the Hamiltonians for which the uniqueness for \eqref{E_CE} holds in the class of $L^\infty$ solutions, 
and therefore the uniqueness for the regular Lagrangian flow, including in particular the case of BV vector fields.

It is worth to mention that, before the general result in \cite{BB_Bressan} was available, the approach introduced above allowed to obtain in \cite{BBG_NI2d} a simpler and more direct proof of
the uniqueness of regular Lagrangian flow for nearly incompressible vector fields with bounded variation; see also \cite{BG_steadyNI} for the intermediate step of steady nearly incompressible vector fields, namely vector fields satisfying Def. \ref{D_NI} with $\rho$ constant in time.

The approximate differentiability of the flow has been obtained for autonomous divergence free vector field $b \in \BV(\R^2;\R^2)$ in \cite{BM_Lusin-Lip}, as a consequence of a suitable Lusin-Lipschitz property.

In the present paper we investigate under which assumptions the regular Lagrangian flow inherits the Sobolev or BV regularity of 
the vector field.
The first result is a local estimate for nearly incompressible vector fields.

\begin{proposition}\label{P_local}
Let $b\in \BV(\R^2;\R^2)$ be a bounded nearly incompressible vector field and let $\Omega\subset \R^2$ be an open ball of radius $R>0$ such that there exist $\delta>0$ and $e\in \S^1$ for which $b\cdot e >\delta$ a.e. in $\Omega$.
Let $\Omega' \subset \Omega$ be an open set and $\bar t>0$ be such that $\dist(\Omega', \partial \Omega)>\|b\|_{L^\infty}\bar t$.
Then 
\begin{equation*}
X(\bar t) \in \BV(\Omega').
\end{equation*}
Moreover, if $b \in W^{1,p}(\R^2;\R^2)$ for some $p\ge 1$, then
\begin{equation*}
X(\bar t) \in W^{1,p}(\Omega').
\end{equation*}
\end{proposition}

The following global result is stated for divergence-free vector fields and we additionally assume that the vector field $b \in \BV(\R^2;\R^2)$ is continuous. Since we are going to consider bounded vector fields, by finite speed of propagation, it is not restrictive to assume that $b$ has compact support.
In particular there exists a unique Hamiltonian $H \in C^1_c(\R^2)$ satisfying \eqref{E_Hamiltonian} and it is straightforward to check that the set of critical values
\begin{equation*}
\mathcal S:= \{ h \in \R: \exists x \in \R^2 \left( H(x)=h \mbox{ and }b(x)=0 \right)\}
\end{equation*}
is closed.
Therefore the set of regular values $\mathcal R := H(\R)\setminus \mathcal S$ and $\Omega = H^{-1}(\R\setminus \mathcal S)= H^{-1}(\mathcal R)$ are open.
\begin{theorem}\label{T_global}
Let $b \in \BV(\R^2;\R^2)$ be a continuous divergence-free vector field with bounded support and let $\Omega$ be defined as above.
Then for every $t>0$ the regular Lagrangian flow has a representative  
\begin{equation*}
X(t) \in C^0(\Omega) \cap \BV(\Omega).
\end{equation*}
If moreover $b \in W^{1,p}(\R^2;\R^2)$, then $X(t) \in W^{1,p}(\Omega)$.
\end{theorem}  

The last result is an example that shows that the existence of $\delta>0$ as in Proposition \ref{P_local} cannot be dropped, as well as the restriction to $\Omega$ in Theorem \ref{T_global}. 
\begin{proposition}\label{P_example}
There exists a divergence-free vector field $b:\R^2\to \R^2$ such that $b\in W^{1,p}_\loc(\R^2;\R^2)$ for every $p\in [1,\infty)$, $b(z)\cdot e_1>0$ for $\L^2$-a.e. $z\in \R^2$  and for every time $ t>0$ 
the regular Lagrangian flow
\begin{equation*}
X(t) \notin \BV_\loc(\R^2;\R^2).
\end{equation*}
\end{proposition}
The construction of the Hamiltonian $H$ associated to $b$ in Proposition \ref{P_example} is a suitable modification of the construction in 
\cite{ABC2} of a Lipschitz Hamiltonian for which the uniqueness of the corresponding regular Lagrangian flow fails.
As opposed to the already mentioned result in \cite{Jabin_example}, the proposed construction is deterministic and disproves the Sobolev 
regularity of the regular Lagrangian flow also for autonomous vector fields.
 
We finally mention that the question about the Sobolev or BV regularity of the regular Lagrangian flow associated to autonomous planar vector 
fields was posed to the author by M. Colombo and R. Tione, motivated by the study of the commutativity property of the flows 
associated to vector fields with vanishing Lie bracket \cite{CT_Lie}.

\section{Local estimate for nearly incompressible vector fields}
In this section we prove Proposition \ref{P_local}.
We begin with two preliminary lemmas about autonomous nearly incompressible vector fields in $\R^d$.
In the first lemma we show that in the case of autonomous nearly incompressible vector fields we can assume without loss of generality that the existence 
time $T$ of $\rho$ in Definition \ref{D_NI} is arbitrarily large.
\begin{lemma}\label{L_all_T}
Let $b:\R^d \to \R^d$ be an autonomous nearly incompressible vector field and let  $\rho:[0,T]\times \R^d\to \R$, $C>0$ be as in Definition \ref{D_NI}.
Then there exists $\tilde \rho \in  C^0([0,+\infty); L^\infty_w(\R^d))$ solving \eqref{E_CE} such that for $\L^{d+1}$-a.e. $(t,x) \in \R^+\times \R^d$ it holds
\begin{equation}\label{E_allT}
\tilde C^{-1} \le \tilde \rho(t,x) \le \tilde C, \qquad \mbox{with} \qquad \tilde C = \tilde C(t):= C^{\frac{2t}{T}+1}.
\end{equation}
\end{lemma}

\begin{proof}
By Ambrosio's superposition principle (see \cite{AC_bologna}), there exists a Radon measure $\eta$ on $\Gamma_T:=C([0,T];\R^d)$ such that for every $t \in [0,T]$ it holds
\begin{equation*}
(e_t)_\sharp \eta = \rho (t,\cdot)\L^d,
\end{equation*}
where $e_t(\gamma):=\gamma(t)$ denotes the evaluation map at time $t$ defined on $\Gamma_T$.
We denote by $\{\eta_x\}_{x\in \R^d} \subset \mathcal P(\Gamma_T)$ its disintegration with respect to the evaluation map at time 0, so that
\begin{equation*}
\eta= \int_{\R^d}\rho(0,x)\eta_x dx
\end{equation*}
and we define
\begin{equation*}
\eta'= \int_{\R^d}\rho(T,x)\eta_xdx.
\end{equation*}
Since $\rho \in [C^{-1},C]$, it holds $C^{-2}\rho(0,x)\le \rho(T,x) \le C^2\rho(0,x)$ for $\L^d$-a.e. $x \in \R^d$.
In particular $C^{-2}(e_t)_\sharp \eta \le (e_t)_\sharp \eta' \le C^2 (e_t)_\sharp \eta$ for every $t\in [0,T]$, therefore
\begin{equation}\label{E_est_rho'}
(e_t)_\sharp \eta' = \rho'(t,\cdot) \L^d, \qquad \mbox{with}\qquad C^{-2}\rho(t,\cdot) \le \rho'(t,\cdot) \le C^2\rho(t,\cdot).
\end{equation}
Let $\tilde \rho : (0,2T)\times \R^d \to \R$ be defined by
\begin{equation*}
\tilde \rho (t,z) = \begin{cases}
\rho(t,z) & \mbox{if }t \in (0,T], \\
\rho'(t-T,z) & \mbox{if }t \in (T,2T).
\end{cases}
\end{equation*}
Since $\tilde \rho$ solves \eqref{E_CE} in $\D'((0,T)\times \R^d)$ and  $\D'((T,2T)\times \R^d)$ separately and
$t\mapsto \tilde \rho(t)$ is continuous with respect to the weak* topology in $L^\infty(\R)$, then $\tilde \rho$ solves \eqref{E_CE} in $\D'((0,2T)\times \R^d)$.
By \eqref{E_est_rho'} it follows that for every $t \in [T,2T]$ it holds
\begin{equation}\label{E_-T}
C^{-2}\tilde \rho (t - T, \cdot)\L^d \le \tilde \rho(t,\cdot)\L^d \le C^2 \tilde \rho (t-T,\cdot)\L^d.
\end{equation}
Iterating the construction above we obtain a solution $\tilde \rho:\R^+\times \R^d \to \R$ of \eqref{E_CE} such that \eqref{E_-T} holds for every $t\ge T$. 
In particular for every $N\in \N$ and for every $t \in [NT, (N+1)T]$ it holds
\begin{equation*}
C^{-2N}\rho (t-NT,\cdot) \L^d \le \tilde \rho(t,\cdot)\L^d \le C^{2N}\rho(t-NT,\cdot)\L^d,
\end{equation*}
which immediately implies \eqref{E_allT} since $\rho \in [C^{-1},C]$.
\end{proof}

The vector fields for which the function $\rho$ in Definition \ref{D_NI} can be chosen independent of $t$ are called \emph{steady nearly incompressible}. 
Although not every nearly incompressible autonomous vector field is steady nearly incompressible, we can reduce to the latter case under the assumptions of Proposition \ref{P_local}.
The proof of the following lemma is an adaptation of the argument in \cite{BBG_NI2d}.

\begin{lemma}\label{L_steady}
Let $b:\R^d \to \R^d$ be an autonomous, bounded, nearly incompressible vector field and let $\Omega\subset \R^d$ be an open ball of radius $R>0$.
Assume that there exist $\delta>0$ and $e \in \S^{d-1}$ for which for $\L^d$-a.e. $x\in \Omega$ it holds $b(x)\cdot e \ge \delta$.
Then $b\llcorner \Omega$ is steady nearly incompressible, namely there exists $r:\Omega \to \R$ and $\tilde C>0$ such that 
\begin{equation*}
\tilde C^{-1}\le r \le \tilde C \qquad \mbox{and} \qquad \div (rb)=0 \quad \mbox{in } \Omega.
\end{equation*}
\end{lemma}
\begin{proof}
Let $\rho:(0,T)\times \R^d \to \R$ and $C>0$ be as in Definition \ref{D_NI}. 
Let $\eta \in \M(\Gamma_T)$ the Radon measure provided by Ambrosio's superposition principle.
In particular if we denote by 
\begin{equation*}
\begin{split}
\tilde e: \Gamma_T\times [0,T] & \to [0,T]\times \R^d \\
(\gamma,t) & \mapsto (t,\gamma(t))
\end{split}
\end{equation*}
it holds
\begin{equation*}
\tilde e_\sharp (\eta \times \L^1) = \rho \left(\L^1 \times \L^d\right).
\end{equation*}
For every $\gamma\in \Gamma$ we set
\begin{equation*}
I_\gamma := \{t \in [0,T]: \gamma(t)\in \Omega\}.
\end{equation*} 
Let $I_{\gamma,0}=[0,t^-_\gamma)$ be the (possibly empty) connected component of $I_\gamma$ containing $0$ and similarly let $I_{\gamma,T}=(t^+_\gamma,T]$ be the connected component of $I_\gamma$ containing $T$.
We denote by
\begin{equation*}
\tilde I_\gamma = [0,T] \setminus (I_{\gamma,0} \cup I_{\gamma,T})
\end{equation*}
and 
\begin{equation*}
\Gamma^-:= \{ \gamma \in \Gamma_T : \gamma(0)\in \Omega)\}, \qquad \Gamma^+:= \{ \gamma \in \Gamma_T : \gamma(T)\in \Omega)\}.
\end{equation*}
Moreover we consider
\begin{equation*}
\tilde \eta = \eta \otimes (\L^1\llcorner \tilde I_\gamma).
\end{equation*}
By definition $\tilde \eta \le \eta \times \L^1$ therefore there exists $\tilde \rho \in L^\infty((0,T)\times \R^d)$ such that  
\begin{equation*}
\tilde e_\sharp \tilde \eta = \tilde \rho (\L^1 \times \L^d).
\end{equation*} 
with $0\le \tilde \rho \le \rho$. 
The following standard computation shows that the density $\tilde \rho$ satisfies the continuity equation
\begin{equation*}
\partial_t \tilde \rho + \div_x(\tilde \rho b) = \mu \qquad \mbox{in }\D'((0,T)\times\R^d),
\end{equation*}
where
\begin{equation*}
\mu = \int_{\Gamma^-}\delta_{t^-_\gamma, \gamma(t^-_\gamma)} d\eta(\gamma) - \int_{\Gamma^+}\delta_{t^+_\gamma, \gamma(t^+_\gamma)} d \eta (\gamma).
\end{equation*}
Given $\varphi \in C^\infty_c((0,T)\times \R^d)$ it holds
\begin{equation*}
\begin{split}
\langle \partial_t \tilde \rho + \div_x(\tilde \rho b), \varphi \rangle = &~ - \int \tilde \rho(t,x) (\varphi_t(t,x) + b(t,x)\cdot \nabla_xb(t,x))dxdt \\
=&~  -  \int (\varphi_t(t,\gamma(t)) + b(t,\gamma(t))\cdot \nabla_xb(t,\gamma(t))) \chi_{\tilde I_\gamma}(t)dt d\eta (\gamma) \\
=&~  -  \int (\varphi_t(t,\gamma(t)) + \dot\gamma(t)\cdot \nabla_xb(t,\gamma(t))) \chi_{\tilde I_\gamma}(t)dt d\eta (\gamma) \\
= &~ - \int \frac{d}{dt}(\varphi(t,\gamma(t))\chi_{\tilde I_\gamma}(t)dt d\eta (\gamma) \\
= &~ \int \left( \varphi(t^-_\gamma,\gamma(t^-_\gamma))- \varphi(t^+_\gamma,\gamma(t^+_\gamma))\right) d\eta(\gamma) \\
= &~ \int \varphi d \mu,
\end{split}
\end{equation*}
where in the last equality we used that $\varphi (0,\cdot)=\varphi(T,\cdot)\equiv 0$.
In particular $\mu$ is concentrated on $\partial \Omega$ so that 
\begin{equation}\label{E_CE_Omega}
\partial_t \tilde \rho + \div_x(\tilde \rho b) = 0 \qquad \mbox{in }\D'((0,T)\times\Omega).
\end{equation}

Since $b\cdot e >\delta$ in $\Omega$, every connected component of $I_\gamma$ has length at most $2R/\delta$.
Up to change the constant $C>0$, by Lemma \ref{L_all_T} we can assume that
\begin{equation*}
T\ge \frac{6R}{\delta},
\end{equation*}
therefore it follows that $\tilde \rho(t,z) = \rho(t,z)$ for $\L^1\times \L^d$-a.e. $(t,z) \in [T/3,2T/3] \times \Omega$, in particular
\begin{equation}\label{E_lower}
\tilde \rho (t,z) \ge C^{-1} \qquad \mbox{in }[T/3, 2T/3] \times \Omega.
\end{equation}
Being $\tilde \rho(0,x)=\tilde \rho (T,x)=0$ for $\L^d$-a.e. $x \in \Omega$, by integrating \eqref{E_CE_Omega} with respect to $t$, it follows that
\begin{equation*}
r(x) := \frac{1}{T}\int_0^T \tilde \rho(t,x)dx
\end{equation*}
satisfies $\div(rb)=0$ in $\D'(\Omega)$. %(see \cite{BBG} for the details).
From \eqref{E_lower} and the definition of $r$, it follows that for $\L^d$-a.e. $x \in \R^2$ it holds
\begin{equation*}
\frac{1}{3 C} \le r(x) \le  \|\rho\|_{L^\infty} \le  C
\end{equation*} 
and this proves the claim with $\tilde C = 3C$.
\end{proof}

In the following of this paper we restrict to the case $d=2$ and in the remaining part of this section  we will always assume that the hypothesis in Proposition \ref{P_local} are satisfied. 
In particular there exists a Lipschitz Hamiltonian $H:\Omega \to \R$ such that
\begin{equation}\label{E_Hamilton}
rb=\nabla^\perp H \qquad \L^2\mbox{-a.e. in }\Omega.
\end{equation}
The generic point in $\R^2$ will be denoted by $z=(x,y)$ and we assume without loss of generality that $e=e_1$. Being $b\cdot e_1>\delta$ and $r \in [\tilde C^{-1},\tilde C]$ 
for every $h \in H(\Omega)$ there exist an open set $O_h \subset \R$ and a Lipschitz function $f_h:O_h \to \R$ such that 
\begin{equation*}
\{z\in \Omega: H(z)=h\} = \{ (x,y) : x \in O_h, y = f_h(x)\}.
\end{equation*}
We will also denote by
\begin{equation}\label{E_tildef}
\tilde f_h(x) := (x,f_h(x))
\end{equation}
for every $x\in O_h$.
The Lipschitz constant $L$ of $f_h$ can be estimated by
\begin{equation}\label{E_L}
L \le \frac{\|rb\|_{L^\infty}}{\inf_{\Omega} (rb\cdot e_1)}\le \frac{C^2\|b\|_{L^\infty}}{\delta}.
\end{equation}

In the following we consider vector fields $b$ with bounded variation. 
As already mentioned in the introduction, the uniqueness problem for the regular Lagrangian flow associated to $rb$ was solved in \cite{BBG_NI2d}, where in particular it is proven that the local Hamiltonian is preserved by the flow, namely
\begin{equation*}
H(t,X(t,z))=H(z) \qquad \forall z \in \Omega \mbox{ and  } t< \dist(\partial\Omega, z).
\end{equation*}
We will consider the representative of the regular Lagrangian flow defined as follows:
for every $z = (x,y) \in \Omega$ and $t< \dist(\partial\Omega, z)$, we set $X(t,z)=(X_1(t,z),f_{H(z)}(X_1(t,z)))$ where $X_1(t,z)$ is uniquely determined by
\begin{equation}\label{E_precise}
\int_{x}^{X_1(t,z)}\frac{1}{\tilde b_1(s, f_{H(z)}(s))}ds = t,
\end{equation}
and where $\tilde b$ denotes the precise representative of $b$, defined at $\H^1$-a.e. $z \in \R^2$ (see for example \cite{AFP_book}).
In particular, if we denote by
\begin{equation*}
R:=\{h\in H(\Omega): |Db|(H^{-1}(h)\cap \Omega)=0\},
\end{equation*}
it holds that $\H^1$-a.e. $z \in H^{-1}(h)\cap \Omega$ is a Lebesgue point of $b$ with value $\tilde b(z)$. In the following we will still denote by $b$ the precise representative $\tilde b$.

\begin{proposition}\label{P_Lip_NI}
Let  $b \in \BV(\R^2;\R^2)$ be a bounded autonomous nearly incompressible vector field. 
Let $\Omega\subset \R^2$ be an open ball of radius $R>0$ such that there exist $\delta>0$ and $e\in \S^1$ for which $b\cdot e >\delta$ a.e. in $\Omega$.
Then there exists $g \in \BV_\loc(\R)$ and a constant $C'=C'(R,\|b\|_{L^\infty},\delta, C)>0$ such that for every $z,z' \in \Omega$ with $H(z),H(z')\in R$ and every $\bar t >0$ for which
\begin{equation*}
\dist (z,\partial \Omega), \dist (z',\partial \Omega) > \|b\|_{L^\infty} \bar t,
\end{equation*}
it holds
\begin{equation}\label{E_Lip_loc}
|X(\bar t,z)-X(\bar t,z')|\le C' \big(|z-z'|  + |g(H(z))-g(H(z'))|\big),
\end{equation}
where $C$ is the compressibility constant in Definition \ref{D_NI} and $H$ is the Hamiltonian introduced in \eqref{E_Hamilton}.

If moreover 
 $b\in W^{1,p}(\R^2;\R^2)$ for some $p\in [1,\infty)$, then \eqref{E_Lip_loc} holds with $g \in W^{1,p}_\loc(\R)$.
\end{proposition}

\begin{proof}
We denote by $h=H(z)$ and $h'=H(z')$. By \eqref{E_precise} it follows that 
\begin{equation}\label{E_period}
\int_{x}^{X_1(\bar t,z)}\frac{1}{b_1(\tilde f_{h}(s))}ds = \bar t = \int_{x'}^{X_1(\bar t,z')}\frac{1}{b_1(\tilde f_{h'}(s))}ds, 
\end{equation}
where $\tilde f_h$ and $\tilde f_{h'}$ are defined in \eqref{E_tildef}.
%We denote by 
%\begin{equation*}
%f_1(x)= \frac{1}{b_1(z_{h_1}(x))},\qquad f_2(x)= \frac{1}{b_1(z_{h_2}(x))}.
%\end{equation*}
Without loss of generality we assume $x\le x'$ and we also suppose that $X_1(\bar t, z)\le X_1(\bar t, z')$, being the opposite case analogous.
We first estimate the distance of the horizontal components of the flows.
We denote by
\begin{equation*}
I_1=(x,x'), \quad I_2 = (x', X_1(\bar t, z)), \quad I_3=(X(\bar t, z), X(\bar t,z')).
\end{equation*}
If $I_2 = \emptyset$, since $b \cdot e_1>\delta$ in $\Omega$, then
\begin{equation*}
|z'-z|\ge |x'-x|\ge |X_1(\bar t,z) - x| \ge \bar t \delta, 
\end{equation*}
therefore
\begin{equation*}
\begin{split}
|X_1(\bar t, z')-X_1(\bar t, z)| & \le ~ |X_1(\bar t, z')-x'| + |x'-x| + |x- X_1(\bar t, z)| \\
& \le ~ \|b\|_{\infty} \bar t + |x'-x| + \|b\|_{\infty} \bar t \\
& \le ~ \left( \frac{2\|b\|_{\infty}}{\delta} + 1\right) |x'-x|.
\end{split}
\end{equation*}
If $I_2\ne \emptyset$, it follows by \eqref{E_period} that
\begin{equation*}
\begin{split}
|X_1(\bar t, z') - X_1(\bar t, z)| & \le ~ \|b\|_{\infty}\int_{I_3}\frac{1}{b_1(\tilde f_{h'}(s))}ds \\
& = ~ \|b\|_{\infty}\left( \int_{I_1}\frac{1}{b_1(\tilde f_{h}(s))}ds + \int_{I_2}\frac{1}{b_1(\tilde f_{h}(s))}ds - \int_{I_2} \frac{1}{b_1(\tilde f_{h'}(s))}ds\right) \\
& \le ~ \frac{\|b\|_{\infty}}{\delta}|x'-x| + \|b\|_{\infty}\int_{I_2}\left|\frac{1}{b_1(\tilde f_{h}(s))}-\frac{1}{b_1(\tilde f_{h'}(s))}\right| ds\\
& \le ~ \frac{\|b\|_{\infty}}{\delta}|x'-x| + \frac{\|b\|_{\infty}}{\delta^2}|Db| (H^{-1}(I(h,h'))),
\end{split}
\end{equation*}
where $I(h,h')$ denotes the closed interval with endpoints $h$ and $h'$; in the last inequality we used that the function $v\mapsto 1/v$ is $\delta^{-2}$-Lipschitz on $(\delta,+\infty)$.
Then we estimate the difference of the vertical components:
\begin{equation}\label{E_vert}
|X_2(\bar t, z')-X_2(\bar t, z)|  \le ~ |X_2(\bar t, z') - f_{h'}(X_1(\bar t, z))| + |f_{h'}(X_1(\bar t, z)) - X_2(\bar t, z)|.
\end{equation}
By definition of $f_{h'}$ it holds
\begin{equation}\label{E_vert1}
|X_2(\bar t, z') - f_{h'}(X_1(\bar t, z))| = |f_{h'}(X_1(\bar t, z')) - f_{h'}(X_1(\bar t, z))| \le L |X_1(\bar t, z') - X_1(\bar t, z)|,
\end{equation}
where $L$ denotes the Lipschitz constant of the function $f_{h'}$ and is bounded by $\frac{C^2\|b\|_{L^\infty}}{\delta}$ as in \eqref{E_L}.
By definition of $f_h$ we have that
\begin{equation}\label{E_vert2}
\begin{split}
 |f_{h'}(X_1(\bar t, z)) - X_2(\bar t, z)| = & ~  |f_{h'}(X_1(\bar t, z)) - f_{h}(X_1(\bar t, z))| \\
 \le & ~ \frac{|h'-h|}{\inf_\Omega |\partial_2  H|} \\
 \le & ~ \frac{C|h'-h|}{\delta} \\
 \le &~ \frac{C^2\|b\|_{L^\infty}|z'-z|}{\delta}
 \end{split}
\end{equation}
where we used $b_1\ge \delta$, $\partial_y H = rb_1$ and $\|\nabla H\|_{L^\infty}\le C\|b\|_{L^\infty}$.
Plugging \eqref{E_vert1} and \eqref{E_vert2} in \eqref{E_vert}, we finally obtain
\begin{equation*}
|X_2(\bar t, z')-X_2(\bar t, z)| \le 
\frac{C^2 \|b\|_{L^\infty}}{\delta} \left[ \left(1 + \frac{2\|b\|_{L^\infty}}{\delta}\right)|z'-z| 
+ \frac{\|b\|_{L^\infty}}{\delta^2} |Db| (H^{-1}(I(h,h'))) 
\right]
\end{equation*}
so that \eqref{E_Lip_loc} holds with
\begin{equation*}
C'=\frac{C^2\|b\|_{\infty}}{\delta^2}\left(1+\frac{2\|b\|_{L^\infty}}{\delta}+\frac{\|b\|_{L^\infty}}{\delta^2}\right), \qquad g(h)=  |Db| (\{H\le h\}).
\end{equation*}
Notice that $g \in \BV_\loc(\R)$ by construction, since $Dg = H_\sharp |Db|$ is a finite Radon measure.

If $b\in W^{1,p}(\R^2;\R^2)$ the same computation leads to \eqref{E_Lip_loc} with
\begin{equation*}
g(h) = \int_{\{H\le h\}}|Db|(z)dz.
\end{equation*}
It only remains to check that $g \in W^{1,p}_\loc(\R)$.
Denoting by $E_h=\{z\in \Omega:  H(z)=h\}$, by the coarea formula we have that 
\begin{equation*}
|\nabla H| \L^2\llcorner \Omega = \int_\R \H^1\llcorner E_h dh \qquad \mbox{so that} \qquad \L^2 = \int_\R \frac{1}{|\nabla H|}\H^1\llcorner E_h dh
\end{equation*}
and therefore
\begin{equation*}
g'(h) =  \int_{ E_h } \frac{|Db|}{|\nabla H|} d\H^1.
\end{equation*}
Being $|\nabla  H| = |r b | >\delta/C$, then by Jensen's inequality and co-area formula we get
\begin{equation}\label{E_gp}
\begin{split}
\int |g'|^p = &~ \int_\R \left| \int_{E_h} \frac{|Db|}{|\nabla  H|} d\H^1\right|^p dh \\
\le &~ \int_\R \frac{(C\H^1 (E_h))^{p-1}}{\delta^{p-1}}\int_{ E_h}  \frac{|Db|^p}{|\nabla  H|} d\H^1 dh \\
\le &~ \left(\frac{2C\sqrt{1+L^2}R}{\delta}\right)^{p-1} \int_\Omega |Db|^p,
\end{split}
\end{equation}
where $L$ is as above. This concludes the proof of the proposition.
\end{proof}

In order to conclude the proof of Proposition \ref{P_local}, we deduce in the following two lemmas the BV and Sobolev regularity of the flow from the pointwise estimate obtained in Proposition \ref{P_Lip_NI}.
\begin{corollary}\label{C_BV}
In the same setting as in Proposition \ref{P_Lip_NI} let $\Omega' \subset \Omega$ be an open set and $\bar t>0$ be such that $\dist(\Omega', \partial \Omega)>\|b\|_{L^\infty}\bar t$.
Then 
\begin{equation*}
X(\bar t) \in \BV(\Omega').
\end{equation*}
\end{corollary}
\begin{proof}
From Proposition \ref{P_Lip_NI} it is sufficient to check that $g\circ H\in \BV(\Omega')$. Let $g_n$ be a sequence of smooth functions converging to $g$ in $L^1(\R)$ with $\TV_\R (g_n) \le \TV_\R (g)$.
By coarea formula 
\begin{equation*}
 H_\sharp (\L^2\llcorner \Omega') = \rho \L^1, \qquad \mbox{with} \quad\rho(h) = \int_{E_h}\frac{1}{|\nabla  H|}d\H^1.
\end{equation*}
In particular 
\begin{equation*}
\rho(h) \le \frac{C\H^1(E_h)}{\delta} \le \frac{2RC\sqrt{1+L^2}}{\delta}
\end{equation*}
is uniformly bounded. 
Hence $g_n\circ  H$ converges in $L^1(\Omega')$ to $g\circ  H$ and
\begin{equation*}
\begin{split}
\TV_{\Omega'} (g \circ  H) \le &~ \liminf_{n\to \infty} \TV_{\Omega'}(g_n\circ  H) \\
= &~  \liminf_{n\to \infty} \int_{\Omega'} |g_n'( H(z))\nabla  H(z)| dz \\
\le &~ \liminf_{n\to \infty} \|\nabla  H\|_{L^\infty} \int |g_n'(h)|\rho(h)dh \\
\le &~  \|\nabla  H\|_{L^\infty} \|\rho\|_{L^\infty} \TV_\R(g). \qedhere
\end{split}
\end{equation*}
\end{proof}

\begin{corollary}\label{C_W1p}
Let us consider the same setting as in Proposition \ref{P_Lip_NI} with $b\in W^{1,p}(\R^2;\R^2)$. Let $\Omega' \subset \Omega$ be an open set and $\bar t>0$ be such that $\dist(\Omega', \partial \Omega)>\|b\|_{L^\infty}\bar t$. Then
\begin{equation*}
X(\bar t) \in W^{1,p}(\Omega').
\end{equation*}
\end{corollary}
\begin{proof}
From Proposition \ref{P_Lip_NI} it is sufficient to check that $g\circ  H \in W^{1,p}(\Omega')$. By chain rule and coarea formula we have
\begin{equation*}
\begin{split}
 \int_{\Omega'} |D (g\circ  H)|^p dz \le &~ \int_{\Omega'}|g' \circ  H|^p |\nabla  H|^p dz \\
 \le  &~ (C\|b\|_{\infty})^{p-1} \int_{\Omega'}|g' \circ  H|^p |\nabla  H| dz \\
 \le  &~ (C\|b\|_{\infty})^{p-1} \int_\R \int_{E_h}|g'\circ  H|^pd \H^1dh \\
 =  &~ (C\|b\|_{\infty})^{p-1} \int_\R |g'|^p(h) \H^1(E_h)dh \\
 \le &~ 2\sqrt{1+L^2}R(C\|b\|_{\infty})^{p-1}\int_\R |g'|^p(h)dh. \qedhere
 \end{split}
\end{equation*}
\end{proof}

\section{Global estimate for divergence free vector fields}
In this section we prove Theorem \ref{T_global}.
In the next lemma we show that we can cover $\Omega$ with countably many open sets invariant for the flow and such that $|b|$ is uniformly bounded from below far from 0.
\begin{lemma}\label{L_Omega_k}
Let $b$ and $\Omega$ as in Theorem \ref{T_global} and let $H \in C^1_c(\R^2)$ be the Hamiltonian associated to $b$ as in \eqref{E_Hamiltonian}.
For every $k\in \N$ let $\Omega_k := H^{-1}(\{ h: \min_{H^{-1}(h)}|b|>1/k \})$. Then $\Omega_k$ is open,
\begin{equation*}
\overline \Omega_k \subset \Omega_{k+1}, \qquad \mbox{and} \qquad \Omega = \bigcup_{k\in \N} \Omega_k.
\end{equation*}
\end{lemma}
\begin{proof}
The last equality follows immediately from the definition of $\Omega$ and the continuity of $b$. In order to complete the proof it is sufficient to check that
the map 
\begin{equation*}
h \mapsto  \min_{H^{-1}(h)}|b|
\end{equation*}
is continuous on the set $\mathcal R$ of regular values of $H$. 
Both the lower semicontinuity and the upper semicontinuity are straightforward consequences of the continuity of $b$ and the compactness of the level sets $H^{-1}(h)$ with regular value $h$.
\end{proof}

The main estimate in the proof of Theorem \ref{T_global} is proven in the following lemma.

\begin{lemma}\label{L_main}
Let $H \in C^1_c(\R^2)$ be such that $b := \nabla ^\perp H \in \BV(\R^2)$ and let $k\in \N$ and $\Omega_k\subset \R^2$ be as above.
Then there exist a representative of the regular Lagrangian flow $X$, $g \in \BV_\loc \cap C^0(\R)$ and $r>0$ such that for every $t>0$ the following holds:
there exist $c_1>0$ and $c_2>0$ such that $\forall \bar z \in \Omega _k$ and every $z \in B_r(\bar z)$ there exists $s>0$ such that
\begin{enumerate}
\item $|X(t,\bar z) -X(s,z)| \le c_1 |H(\bar z) - H(z)|$ \\
\item $ |t-s| \le c_2 \left( |g(H(\bar z)) - g(H(z))| + |\bar z -z|\right)$.
\end{enumerate}
\end{lemma}
\begin{proof}
The proof is divided in several steps.

\noindent
\emph{Step 1}.
By Lemma \eqref{L_Omega_k} $\Omega_k$ is compactly contained in the open set $\Omega_{k+1}$. Since $b$ is continuous and uniformly bounded from below on $\Omega_{k+1}$, for every $L>0$ there exist $\bar r>0$ and a finite covering $(B_{\bar r}(z_i))_{i=1}^N$ of $\Omega_k$ such that
\begin{enumerate}
\item for every $i=1,\ldots, N$ it holds $B_{4\bar r}(z_i)\subset \Omega_{k+1}$;
\item for every $i=1,\ldots, N$ there exists $e_i\in \S^1$ such that 
\begin{equation}\label{E_flat}
b(z) \cdot e_i \ge |b(z)| \cos (\tan^{-1}(L)) \qquad \forall z \in B_{4\bar r}(z_i).
\end{equation}
\end{enumerate}
We take $L>0$ sufficiently small so that $\cos(\tan^{-1}(L))> 1/2$ and such that
for every $i=1,\ldots, N$ and for every $h \in H(B_{3\bar r}(z_i))$ there exist an open interval $I_{i,h}\subset \R$ 
and a $L$-Lipschitz function $f_{i,h}:I_{i,h}\to \R$ such that 
\begin{equation*}
H^{-1}(h) \cap B_{4\bar r}(z_i) = \left\{ z \in \R^2: z\cdot e_i \in I_{i,h}, \quad z\cdot e_i^\perp = f_{i,h}(z\cdot e_i) \right\}.
\end{equation*}

\noindent
\emph{Step 2}.
We show that the function $g:\R\to \R$ defined by
\begin{equation*}
g(h):= |Db|(\{H\le h\}\cap \Omega_{k+1})
\end{equation*} 
is continuous and with bounded variation. 

Since $Dg= H_\sharp |Db|\llcorner \Omega_{k+1}$ is a finite measure, the function $g$ has bounded variation.
In order to prove that $g$ is continuous it is sufficient to check that for every $h\in \R$ it holds
\begin{equation*}
|Db|(H^{-1}(h)\cap \Omega_{k+1})=0.
\end{equation*}
By Step 1, the set $H^{-1}(h)\cap \Omega_{k+1}$ is the union of finitely many Lipschitz curve of finite length. Being $b$ continuous, the measure $|Db|$ vanishes on all sets with finite $\mathcal H^1$ measure (see for example \cite{AFP_book}), and this proves the continuity of $g$.

\noindent
\emph{Step 3}. Given $T>0$ and $\bar z \in \Omega_k$, we denote by 
\begin{equation*}
\tilde N := \left\lceil \frac{T\|b\|_{L^\infty}}{\bar r} \right\rceil, \qquad \mbox{and} \qquad I_j:=\left[ \frac{j-1}{\tilde N}T, \frac{j}{\tilde N}T\right]=:[t_{j-1},t_j] \quad \mbox{for }j=1,\ldots, \tilde N.
\end{equation*}
Moreover  for every $j \in 1,\ldots, \tilde N$ we consider $i=i(j) \in 1,\ldots, N$ such that
\begin{equation*}
X(t_j,\bar z) \in B_{\bar r}(z_{i(j)}).
\end{equation*}
We set
\begin{equation}\label{E_def_r}
r:= \min \left\{ \bar r, \frac{\bar r}{2(k+1)\|b\|_{L^\infty}}, \frac{T}{2\tilde N (k+1)} \right\}.
\end{equation}

For every $z\in \R^2$ with $|z-\bar z |\le r$ we prove that for every $j=1,\ldots, \tilde N$ there exists $s_j >0$ such that
\begin{equation*}
 |X(t_j,\bar z)-X(s_j,z)|\le 2(k+1) |\bar h - h|
\end{equation*}
and
\begin{equation}\label{E_recursion}
|t_j-s_j| \le 2(k+1)|z-\bar z| + j(k+1)^2|g(\bar h) - g(h)| + 2(j-1)(k+1)^2|\bar h -h|.
\end{equation}
By \eqref{E_def_r} and the definition of $(z_{i})_{i=1}^N$, for every $j=1,\ldots, \tilde N$ there exists a unique point $\tilde z_j \in \Omega_{k+1}$ in $z \in B_{2\bar r}(z_{i(j)})$ such that 
\begin{equation*}
H(\tilde z_j) = h \qquad \mbox{and} \qquad \tilde z_j \cdot e_i = X(t_j, \bar z) \cdot e_i.
\end{equation*}
We immediately have
\begin{equation*}
|X(t_j,\bar z) - \tilde z_j| = | ( X(t_j,\bar z) - \tilde z_j)\cdot e_i^\perp| \le \frac{|\bar h - h|}{\min_{B_{2\bar r}(z_{i(j)})}b\cdot e_j}\le 2(k+1)|\bar h - h|.
\end{equation*}
Notice in particular that $|X(t_j,\bar z) - \tilde z_j|\le \bar r$ by \eqref{E_def_r}.
In order to prove the claim, it is sufficient to show that for every $j=1,\ldots, \tilde N$, there exists $s_j\ge 0$ satisfying \eqref{E_recursion} and such that $X(s_j)=\tilde z_j$. We prove this by induction on $j$.

\noindent \emph{Case $j=1$}. 
First we observe that 
\begin{equation*}
\left\{ \bar z \cdot e_{i(1)}, z \cdot e_{i(1)}, X(t_1,\bar z) \cdot e_{i(1)}\right\} \subset I_{i(1),h} \cap I_{i(1),\bar h}.
\end{equation*}
In particular $z$ and $\tilde z_1$ belong to the same connected component of $H^{-1}(h)$.
Since $|z-\bar z|< r$ it trivially holds
\begin{equation*}
 \bar z \cdot e_{i(1)} - r< z \cdot e_{i(1)}< \bar z \cdot e_{i(1)} + r.
\end{equation*}
Moreover 
\begin{equation*}
X(t_1, \bar z) \cdot e_{i(1)} \ge \bar z \cdot e_{i(1)} +  t_1 \min_{B_{4\bar r}(z_{i(1)})}b\cdot e_{i(1)} \ge \bar z \cdot e_{i(1)} + \frac{t_1}{2(k+1)} >  \bar z \cdot e_{i(1)} + r.
\end{equation*}
We assume $ z \cdot e_{i(1)} \ge \bar z \cdot e_{i(1)}$, being the opposite case analogous. We denote by $\tilde t_0\ge 0$ the unique 
$t \in [0,t_1)$ such that $X(t,\bar z) \cdot e_{i(1)}= z \cdot e_{i(1)}$.
We have 
\begin{equation*}
\tilde t_0 \le \frac{|z-\bar z|}{\min_{B_{4\bar r}(z_{i(1)})} b\cdot e_{i(1)}} \le 2(k+1)|z-\bar z|.
\end{equation*}
Moreover
\begin{equation*}
t_1-\tilde t_0= \int_{z\cdot e_{i(1)}}^{\tilde z_1 \cdot e_{i(1)}} \frac{1}{b\cdot e_{i(1)}}(\bar f_{i,\bar h}(x))dx
\end{equation*}
and similarly
\begin{equation*}
s_1 = \int_{z\cdot e_{i(1)}}^{\tilde z_1 \cdot e_{i(1)}}\frac{1}{b\cdot e_{i(1)}}(\bar f_{i, h}(x))dx.
\end{equation*}
We denote by 
\begin{equation*}
S:= \{ z' \in \R^2 : z'\cdot e_{i(1)}\in (z\cdot e_{i(1)}, \tilde z_1 \cdot e_{i(1)}), z'\cdot e_{i(1)} \in
(f_{i(1),h}(z'\cdot e_{i(1)}), f_{i(1),\bar h}(z'\cdot e_{i(1)})  )   \}.
\end{equation*}
Therefore
\begin{equation*}
\begin{split}
|t_1 - s_1| \le &~ |\tilde t_0| + |t_1-\tilde t_0 - s_1| \\
\le &~ 2(k+1)|z-\bar z|+ \int_{z\cdot e_{i(1)}}^{\tilde z_1 \cdot e_{i(1)}} \left| \frac{1}{b\cdot e_{i(1)}}(\bar f_{i,\bar h}(x)) - \frac{1}{b\cdot e_{i(1)}}(\bar f_{i, h}(x)) \right| dx \\
\le &~ 2(k+1)|z-\bar z| + \left(\frac{1}{\inf_{S} b \cdot e_{i(1)}}\right)^2 |Db|(S) \\
\le  &~ 2(k+1) |z-\bar z| + 4(k+1)^2 |g(\bar h)- g(h)|,
\end{split} 
\end{equation*}
where, in the third inequality we used that the maps $v \mapsto 1/v$ is $1/\delta^2$-Lipschitz on $[\delta, +\infty)$ with $\delta= \inf_{S} b\cdot e_{i(1)}$.
This proves \eqref{E_recursion} for $j=1$.

\noindent \emph{Case $j>1$}. We assume 
\begin{equation}\label{E_hyp_rec}
|t_{j-1}-s_{j-1}| \le 2(k+1)|z-\bar z| + (j-1)(k+1)^2|g(\bar h) - g(h)| + 2(j-2)(k+1)^2|\bar h -h|
\end{equation}
and we prove \eqref{E_recursion}. 
We observe that $|X(t_j,\bar z)-X(t_{j-1},\bar z)|\le \bar r$, therefore \eqref{E_flat} implies that $b \cdot e_{i(j)}> 1/(k+1)$ in $B_{2\bar r}(X(t_{j-1},\bar z))$.
In particular we can define $\tilde w_{j-1}$ as the unique $z\in B_{2\bar r}(X(t_{j-1},\bar z))$ such that 
\begin{equation*}
H(z) = h \qquad \mbox{and} \qquad   z \cdot e_{i(j)} = X(t_{j-1},\bar z) \cdot e_{i(j)}.
\end{equation*}
We have
\begin{equation*}
|\tilde w_{j-1}- \tilde z_{j-1}| \le |\tilde w_{j-1}- X(t_{j-1},\bar z)| + |X(t_{j-1},\bar z)- \tilde z_{j-1}| \le 2(k+1)|h-\bar h|.
\end{equation*}
Since $\tilde w_{j-1}\cdot e_{i(j)}, \tilde z_{j-1}\cdot e_{i(j)} \in I_{i(j-1),h}$, then there exists $s_{j-1}>0$ such that
$\tilde w_{j-1}= X(\tilde s_{j-1}, z)$ with 
\begin{equation}\label{E_est1}
|\tilde s_{j-1} - s_{j-1}| \le \frac{2(k+1)}{\min_{B_{4\bar r}(z_{i(j-1)}))} b\cdot e_{i(j-1)}}|\bar h - h| \le 2(k+1)^2 |\bar h -h|.
\end{equation}
By the triangular inequality
\begin{equation}\label{E_triangular}
\begin{split}
|t_j - s_j| \le  &~ | t_j - t_{j-1} + t_{j-1} - s_{j-1} + s_{j-1} - \tilde s_{j-1} + \tilde s_{j-1} - s_j| \\
\le &~ |t_{j-1}-s_{j-1}| + |t_j - t_{j-1}- s_j + \tilde s_{j-1}| + | \tilde s_{j-1}- s_{j-1}|.
\end{split}
\end{equation}
The same computation as in the case $j=1$ gives
\begin{equation}\label{E_est2}
\begin{split}
|t_j - t_{j-1}- s_j + \tilde s_{j-1}| = &~ \left| \int_{X(t_{j-1},\bar z)\cdot e_{i(j)}}^{X(t_{j},\bar z)\cdot e_{i(j)}} 
\left(\frac{1}{b\cdot e_{i(j)}}(\bar f_{i,\bar h}(x)) - \frac{1}{b\cdot e_{i(j)}}(\bar f_{i, h}(x)) \right) dx \right| \\
\le &~ (k+1)^2 |g(\bar h)-g(h)|.
\end{split}
\end{equation}
By plugging \eqref{E_hyp_rec}, \eqref{E_est1} and \eqref{E_est2} into \eqref{E_triangular}, we finally get \eqref{E_recursion}. 

The statement is therefore proven with $s=s_{\tilde N}$, $c_1=2(k+1)$ and $c_2= \tilde N(k+1)^2 (1+2 \|b\|_{L^\infty}) + 2(k+1)$.
\end{proof}

\begin{remark}\label{R_W1p}
If we additionally assume that $b \in W^{1,p}(\R^2)$ in Lemma \ref{L_main}, then the statement holds true with
\begin{equation*}
g(h):= \int_{\{H\le h\}\cap \Omega_{k+1}}|\nabla H| dz.
\end{equation*}
In particular we showed in the proof of Proposition \ref{P_local} that $g \in W^{1,p}_\loc(\R)$.
\end{remark}

\begin{proof}[Proof of Theorem \ref{T_global}.]
By Lemma \ref{L_Omega_k}, it is sufficient to prove that $X(t)\in \BV(\Omega_k)$ for every $k \in \N$.
In the same setting as in Lemma \ref{L_main}, if $\bar z \in \Omega_k$  and $z\in B_r(\bar z)$, then
\begin{equation}\label{E_est_global}
\begin{split}
|X(t,\bar z) - X(t,z)| \le &~ |X(t,\bar z) - X(s, z)| + |X(s,z)- X(t,z)| \\
\le &~ c_1  |H(\bar z) - H(z)| + \|b\|_{L^\infty} |t-s| \\
\le &~ c_1  |H(\bar z) - H(z)| + c_2 \|b\|_{L^\infty} \left( |g(H(\bar z)) - g(H(z))| + |\bar z -z|\right) \\
\le &~ \|b\|_{L^\infty}(c_1+c_2) |\bar z -z| +  c_2 \|b\|_{L^\infty} |g(H(\bar z)) - g(H(z))|.
\end{split}
\end{equation}
The argument in the proof of Corollary \ref{C_BV} shows that $g\circ H \in \BV(\Omega_k)$ therefore it follows from \eqref{E_est_global} that
$X(t) \in \BV(\Omega_k \cap B_r(z))$ for every $z \in \Omega_k$. Being $\Omega_k$ bounded this proves that $X(t) \in \BV(\Omega_k)$.
Finally the continuity of $X(t)$ follows immediately from \eqref{E_est_global} and the continuity of $g$.
If moreover we assume that $b \in W^{1,p}(\R^2)$, then the same argument proves that $X(t) \in W^{1,p}(\Omega_k)$ thanks to Remark \ref{R_W1p} and Corollary \ref{C_W1p}.
\end{proof}

\begin{remark}
By inspection in the argument used to prove Lemma \ref{L_main}, we observe that $\|X(t)\|_{\BV}(\Omega)$ (or $\|X(t)\|_{W^{1,p}}(\Omega)$)  is locally bounded for $t\in [0,+\infty)$ and it diverges at most linearly in $t$ as $t \to \infty$.
\end{remark}

\section{Example}
In this section we prove Proposition \ref{P_example}.

\subsection{Construction of $C_n$, $D_n$, $E_n$ and $F_n$}
We consider the following parameters:
\begin{equation}\label{E_parameters}
c_n = \frac{1}{n^22^n}, \qquad a_n = \frac{n-1}{2n}\left( \frac{c_n}{2}-c_{n+1}\right)\sim \frac{1}{2^{n-1}n^3}, \qquad   r_n = \frac{1}{2n}\left( \frac{c_n}{2}-c_{n+1}\right)\sim  \frac{1}{2^{n-1}n^4}. %= \frac{2n+1}{2^{n+1}n^2(n+1)^2}\sim \frac{1}{2^nn^3}.
\end{equation}
We set $C_1= [0,1/2]^2\subset \R^2$ and we inductively define $C_{n+1}$ for $n\ge 1$ as follows:
$C_{n+1}\subset C_{n}$ and every connected component $R'$ of $C_{n}$ contains two connected components of $C_{n+1}$, which are squares of side $c_n$ as in Figure \ref{F_Cn}. 
For every $n\in \N$ we also consider the sets $D_n, E_n, F_n \subset C_n$ as in Figure \ref{F_Cn}.

\begin{figure}
\centering
\def\svgwidth{\columnwidth}
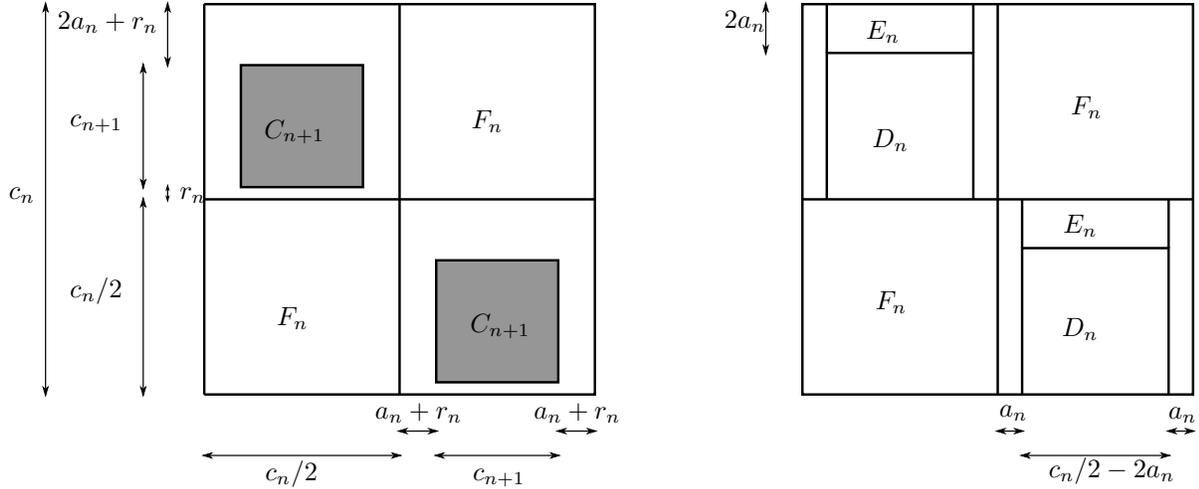
\caption{Construction of the set $C_{n+1}, D_n, E_n, F_n \subset C_n$. The two pictures represent a connected component of $C_n$.}\label{F_Cn}. 
\end{figure}

We observe that for every $n\ge 1$ it holds
\begin{equation*}
C_{n+1}= \{ z \in D_n : \dist (z,\partial D_n) \ge r_n\}.
\end{equation*}

\subsection{Construction of $f_n$ and $h_n$}

The function $f_0:\R^2 \to \R$ is defined by $f_0(x,y)=y$. The function $f_n$ coincides with $f_{n-1}$ on $\R^2 \setminus C_n$ and its level lines in $C_n$ are as in Figure \ref{F_lines}.
In particular $f_n$ coincides with $f_{n-1}$ on $F_n$.

\begin{figure}
\centering
\def\svgwidth{0.6\columnwidth}
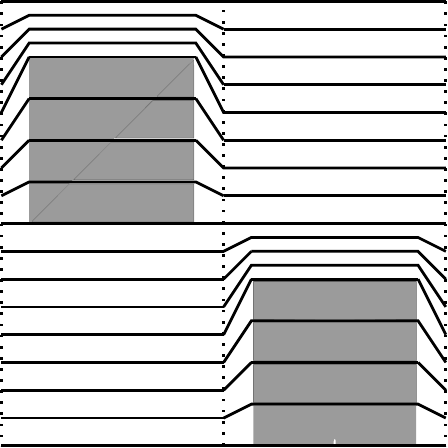
\caption{Level sets of $f_n$ on a connected component $R'$ of $C_n$. The set $D_n\cap R'$ is colored in gray.}\label{F_lines}.
\end{figure}

Let $R$ be a connected component of $D_n$, then $f_n$ is affine on $R$ and depends only on $y$, therefore $\nabla f_n=(0,v_n)$.
Let $R'$ be a connected component of $C_n$ and denote by $s_n=\Osc(f_n,R')$.
\begin{equation*}
v_n = \frac{\Osc(f_n,R)}{\h(R)} = \frac{\Osc(f_n,R')}{\h(R')},
\end{equation*}
where $\h(R)=c_{n+1}+2r_n$, $\h(R')=c_{n+1}$, $\Osc(f_n,R) = s_n/4$ and $ \Osc(f_n,R') = s_{n+1}$ so that
\begin{equation*}
4s_{n+1}= \frac{c_{n+1}}{c_{n+1}+2r_n}s_n.
\end{equation*}
In particular
\begin{equation*}
4^ns_n = 4 c_1 \prod_{l=2}^n\frac{c_l}{c_l+2r_{l-1}} \searrow 4c_1 \prod_{l=2}^\infty \frac{c_l}{c_l+2r_{l-1}}. 
\end{equation*}
From the choice \eqref{E_parameters} it follows that $r_{l-1} c_l^{-1}=O(l^{-2})$ therefore
\begin{equation*}
\log\left(\frac{c_l}{c_l+2r_{l-1}}\right) = O(l^{-2}).
\end{equation*}
In particular the infinite product is strictly positive and we denote it by $\sigma$. We finally get
\begin{equation*}
v_n = \frac{s_{n+1}}{c_{n+1}} \sim c_1\sigma\frac{n^2}{2^n}.
\end{equation*}

Similarly we compute the speed $\nabla f_n = (0,v_n')$ in the region $E_n$ as in the picture.
Denoting by $R''$ one of its components, we have
\begin{equation*}
v'_n  = \frac{\Osc(f_n,R'')}{\h(R'')} = \frac{s_n}{8a_n} \sim \frac{c_1\sigma}{4} \frac{n^3}{2^n}.
\end{equation*}

\subsection{Estimates on the norms of $\nabla f_n$ and $\nabla h_n$}
We first estimate $\|\nabla f_n\|_{L^\infty(C_n)}$. From Figure \ref{F_lines} we observe that $ \|\partial_2 f_n\|_{L^\infty(C_n)} = v'_n$
and the maximal slope of the level sets of $f_n$ in $C_n$ is $\frac{c_n - 8 a_n}{4a_n}$.
Therefore
\begin{equation*}
\|\nabla f_n\|_{L^\infty}(C_n) \le v'_n\left( \frac{c_n - 8 a_n}{4a_n} + 1\right) \sim \frac{c_1\sigma}{16}\frac{n^4}{2^n}.
\end{equation*}
Since $f_n$ and $f_{n-1}$ coincide outside $C_n$, it holds 
\begin{equation*}
\|\nabla h_n\|_{L^\infty}\le \|\nabla f_n\|_{L^\infty(C_n)} +  \|\nabla f_{n-1}\|_{L^\infty(C_{n-1})} = O \left(\frac{n^4}{2^n}\right).
\end{equation*} 
This proves that
\begin{equation*}
f = \lim_{n\to \infty}f_n = f_0 + \sum_{l=1}^\infty h_l
\end{equation*}
is a Lipschitz function.

\subsection{Estimate on crossing time}\label{Ss_crossing}

Let $T^1_n$ be the amount of time needed by an integral curve of the vector field $-\nabla^\perp f_{n-1}$ to cross a connected component of $C_n$.
Then
\begin{equation*}
T^1_n= \frac{c_n}{v_{n-1}}\sim \frac{1}{2c_1\sigma n^4}.
\end{equation*}
Let moreover $T^s_n$ be the amount of time needed by an integral curve of the vector field $-\nabla^\perp f_n$ intersecting $D_n$ to cross a connected component of $C_n$. 
Since $a_n+r_n=o(c_n)$, then $T^s_n$ is asymptotically equivalent at the sum of the amounts of time needed to cross a connected component of $D_n$ and a connected component of $F_n$, namely
\begin{equation*}
T^s_n \sim \frac{1}{2c_1\sigma n^4} + \frac{1}{4c_1\sigma n^4} .
\end{equation*}
Finally let $T^f_n$ be the amount of time needed by an integral curve of the vector field $-\nabla^\perp f_n$ intersecting $E_n$ to cross a connected component of $C_n$:
similarly as above we have the $T^f_n$ is asymptotically equivalent to the amount of  time needed to cross a connected component of $F_n$, namely
\begin{equation*}
T^f_n \sim \frac{1}{4c_1\sigma n^4}.
\end{equation*}
Let us denote by $X_n$ the flow of $-\nabla^\perp f_n$ and by $X$ the flow of $-\nabla^\perp f$. 
For every $z=(x,y) \in \R^2$ with $x<0$ we define $t_1(z)$ as the unique $t>0$ such that $X(t_1(x))\cdot e_1 = 0$ and 
$t_2(z)$ as the unique $t>0$ such that $X(t_1(x))\cdot e_1 = 1$. Since $f(x,y)=y$ for every $(x,y)\in \R^2$ with $x<0$, then the function $z\mapsto t_2(z)-t_1(z)$ depends only on $y$. We therefore set
\begin{equation*}
T(y):= t_2(-1,y)-t_1(-1,y) 
\end{equation*}
and we observe that for every $x<0$ it holds $T(y)=t_2(x,y)-t_1(x,y)$.
Let $z=(x,y) \in \R^2$ be such that $x <0$ and there exists $t>0$ for which $X_n(t,z) \in E_n$. Then, by construction, $X(t,z)=X_n(t,z)$ for every $t>0$ and therefore 
\begin{equation*}
T(y)= T_1 + \sum_{l=2}^{n-1} ( T^s_l - T^1_l) + T^f_n - T^1_n= : T_n
\end{equation*}

By construction there exist $0<y_1<y_2<\ldots < y_{2^{n-1}}<1$ such that for every $x<0$ and every $k \in [1, 2^{n-1}] \cap \N$ 
there exists $t=t(k)>0$ for which
\begin{equation*}
X(t,-1,y_{k}) \in E_n \quad \mbox{if $k$ is even}\qquad \mbox{and} \qquad X(t,-1,y_{2k-1})\in E_{n-1} \quad \mbox{if $k$ is odd}.
\end{equation*}
%\begin{equation*}
%T(y_{2k})=T_n \qquad \mbox{and} \qquad T(y_{2k-1})= T_{n-1}
%\end{equation*}
In particular
\begin{equation*}
\TV_{(0,1)} T \ge 2^{n-1} (T_n- T_{n-1}) = 2^{n-1}(T_n^f - T^1_n + T^s_{n-1}-T^f_{n-1}) \sim \frac{2^n}{8c_1 \sigma n^4}.
\end{equation*}
This shows that $T$ has not bounded variation.

\subsection{Regularity of the flow}\label{Ss_regularity}
We observe that the function $T$ constructed in the previous section is bounded:
indeed 
\begin{equation*}
\sup T = \sup_n T_n\le T_0 + \sum_{n=2}^\infty |T_n-T_{n-1}| \le 1 + C \sum_{n=1}^\infty \frac{1}{4c_1 \sigma n^4} < \infty
\end{equation*}
for some universal constant $C>0$. Since $f(x,y)= y$ for every $(x,y) \in \R^2 \setminus [0,1/2]^2$, for every $t> \sup T$ it holds
\begin{equation*}
X(t,x,y)\cdot e_1= x + 1 + t -T(y) \qquad \forall (x,y) \in (\sup T - t, 0) \times \R.
\end{equation*}
Since $T$ has not bounded variation, then $X(t) \notin \BV((-\e,0)\times (0,1/2))$ for every $\e>0$ and every $t> \sup T$.
If $R=[a,b]\times [c,d]$ denotes a connected component of $C_n$, the same argument as above shows that $X(t)\notin \BV((a-\e,a)\times (c,d))$ for every $\e>0$ and every $t> t_n$ for some $t_n\to 0$ as $n\to \infty$.
In particular $X(t)\notin \BV_\loc(\R^2;\R^2)$ for every $t>0$.
 
\subsection{More regular vector field}
The example constructed above does not prove Proposition \ref{P_example} since the vector field $b=-\nabla^\perp f$ has no Sobolev regularity.
In order to make the vector field more regular, we consider
\begin{equation*}
\tilde f = f_0 + \sum_{l=1}^\infty h_l \ast  \rho_l,
\end{equation*}
where
\begin{equation*}
\rho_l(z)=r_l^{-2} \rho(z/r_l)
\end{equation*}
and $\rho:\R^2 \to \R$ is a positive smooth function such that
\begin{enumerate}
\item $\supp \rho \subset B_{1/2}(0)$;
\item $\int_{\R^2}\rho = 1$ and $\int_{\R^2}z\rho(z)dz=0$. 
\end{enumerate} 

Let us first check that $\tilde f\in W^{1,p}_\loc(\R^2,\R^2)$ for every $p \in [1,\infty)$: indeed
\begin{equation*}
\begin{split}
\|\nabla^2 (h_l\ast \rho_l)\|_{L^p}\le &  ~ \|\nabla h_l \ast \nabla \rho_l\|_{L^p} \\
\le &~ \|\nabla h_l\|_{L^p}\|\nabla \rho_l\|_{L^1} \\
\le &~ \|\nabla h_l\|_{L\infty} \left(\L^2(\supp (\nabla h_l))\right)^{1/p}\|\nabla \rho_l\|_{L^1} \\
\le &~ O(l^4 2^{-l})  \left(\L^2(C_{l-1})\right)^{1/p} O(r_l^{-1}) \\
= &~ O(l^4 2^{-l}) O(2^{-l/p}l^{-4/p}) O(2^ll^4) \\
= &~ O(2^{-l/p}l^{8-\frac{4}{p}}).
\end{split}
\end{equation*}
Being $\|\nabla^2 (h_l\ast \rho_l)\|_{L^p}$ summable, the sequence 
\begin{equation*}
\tilde f_n:=  f_0 + \sum_{l=1}^n h_l \ast  \rho_l
\end{equation*}
 converges to $\tilde f$ in $W^{2,p}_\loc(\R^2)$ for every $p\in [1,+\infty)$.
We now prove that the same argument of Sections \ref{Ss_crossing} and \ref{Ss_regularity} for $f$ can be applied to $\tilde f$. 

Being $f_n$ affine on each connected component of $D_n$, it follows from the properties of the convolution kernel that
$h_n\ast \rho_n (z)= h_n(z)$ for every $z \in D_n$ such that $\dist (z,\partial D_n)> r_n$.
We denote by
\begin{equation*}
\begin{split}
\tilde D_n&:= \{x \in D_n: \dist (x,\partial D_n)>r_n\}, \\
\tilde E_n&:= \{x \in E_n: \dist (x,\partial E_n)>r_n\}, \\
\tilde F_n&:= \{x \in D_n: \dist (x,\partial F_n)>r_n\}.
\end{split}
\end{equation*}
Observe that all the sets above are non-empty by the choice of the parameters \eqref{E_parameters}.
Since $\tilde D_n=C_{n+1}$, we have in particular that $f_n= \tilde f_n$ on the set $C_{n+1}$.
Similarly $h_n\ast \rho_n = h_n$ on $\tilde E_n \cup \tilde F_n$.
As in Section \ref{Ss_crossing}, we denote by $\tilde T^1_n$ the total amount of time needed by an integral curve of the vector field 
$-\nabla^\perp \tilde f_{n-1}$ to cross a connected component of $C_n$. Since $f_{n-1}=\tilde f_{n-1}$ on $C_n$, then $\tilde T^1_n=T^1_n$.
We moreover denote by $\tilde T^s_n$ the amount of time needed by an integral curve of the vector field $-\nabla^\perp \tilde f_n$ intersecting $\tilde D_n$ to cross a connected component of $C_n$.
Since $\tilde f_n = f_n$ in $\tilde D_n \cup \tilde F_n$ it is straightforward to check that $\tilde T^s_n \sim T^s_n$.
Similarly, we denote $\tilde T^f_n$ the amount of time needed by an integral curve of the vector field $-\nabla^\perp\tilde  f_n$ intersecting $\tilde E_n$ to cross a connected component of $C_n$.
Since $\tilde f_n = f_n$ in $\tilde E_n \cup \tilde F_n$ it is straightforward to check that $\tilde T^f_n \sim T^f_n$.

We are now in position to repeat the argument in Sections \ref{Ss_crossing} and \ref{Ss_regularity} and this proves that for every $t>0$ the regular Lagrangian flow $\tilde X$ of the vector field $-\nabla^\perp \tilde f$ satisfies 
\begin{equation*}
\tilde X(t) \notin \BV_\loc(\R^2,\R^2).
\end{equation*}
This concludes the proof of Proposition \ref{P_example}.

\bibliographystyle{alpha}

%\bibliography{Biblioteca}

\end{document}